\documentclass[12pt]{amsart}
\usepackage{amscd}

\usepackage{amssymb}

\newcommand{\bT}{{\mathbb T}} 
\newcommand{\bP}{{\mathbb P}}

\newcommand{\bZ}{{\mathbb Z}}
\newcommand{\bR}{{\mathbb R}}

\newcommand{\bA}{{\mathbb A}}

\newcommand{\bC}{{\mathbb C}}

\newcommand{\bF}{{\mathbb F}}
\newcommand{\bQ}{{\mathbb Q}}

\newcommand{\la}{{\langle}}
\newcommand{\ra}{{\rangle}}

\newtheorem{thm}{Theorem}[section]
\newtheorem{lemma}[thm]{Lemma}
\newtheorem{cor}[thm]{Corollary}
\newtheorem{prop}[thm]{Proposition}

\numberwithin{equation}{section}

\begin{document}

\title[]{Non stable rationality of projective approximations 
for classifying spaces}

 \author{Nobuaki Yagita}

\address{[N. Yagita]
Department of Mathematics, Faculty of Education, Ibaraki University,
Mito, Ibaraki, Japan}


\email{ nobuaki.yagita.math@vc.ibaraki.ac.jp, }


\begin{abstract}

Let $BG$ be the classifying space of an algebraic group $G$
over a subfield $k$ of  $\bC$ of complex numbers.
We compute a new stable birational invariant defined by 
Benoist-Ottem as the difference of two coniveau filtrations  of a
 smooth projective (Ekedahl) approximation $X$ of $BG\times \bP^{\infty}$.
Then  we show (by without and with the unramified cohomology) in many cases $X$ are not stable
 rational. 
\end{abstract}

\maketitle

\section{Introduction}

Let $X$ be a smooth projective variety over $k\subset \bC$. 
The conception of the rationality is how $X$ is near to
some projective space $\bP^n$ over $k$.  
Indeed,  $X$ is called $rational$ if $X$ is birational to a 
projective space $\bP^n$.  A variety $X$ is called $stable$ rational if $X\times \bP^m$ is rational for some $m\ge 0.$
A variety $X$ is called $retract$ rational if
the rational identity map on $X$ is factorized rationally  through a projective space.

 Of course,  the existences and properties of non  these  rationality for $X$ are widely studied by many authors  (see explanations in \cite{Pi}).  
For examples, such projective $X$ which are surface bundles of three (or four)folds are studied detailedly.  These 
examples are computed by often using the unramified cohomology
$H_{ur}^*(X;\bZ/p)$ which is invariant of (retract) rationality.

 There are another examples (exchanging $\bP^n$ by $\bA^n$); the quasi projective 
variety  represented by  the classifying spaces $BG$ of an affine algebraic groups $G$ [Me].

In this paper , we study the similar but different invariant
$DH^*(X)$ for the 
projective approximation $X=X_G$ by Ekedahl for the classifying space $BG\times \bP^{\infty}$.  Note that
stable rationality types of $BG\times \bP^{\infty}$ and its projective approximation are completely different in general.

For example we compare these invariants when $G=SO_{2m+1}$ 
 \[\begin{cases}  DH^*(BG)\qquad is\ not\ defined,\\
H^*_{ur}(BG;\bZ/2)\cong \bZ/2\{1\}, \qquad \\
DH^*(X_G)/2  \supset  \bZ/2\{w_3,w_5,...,w_{2m+1}\}, \\
  
H^*_{ur}(X_G;\bZ/2)\supset   
\bZ/2\{1,w_2,w_4,...,w_{2m}\} .
\end{cases} \]
(The notation $A\{a,b,...\}$ means the $A$-free module generated by $a,b,...$.)

We compute a new stable birational invariant induced from  
Benoist-Otten \cite{Be-Ot} as the difference of the two coniveau filtrations.  For  a fixed prime $p$, define the
stable birational invariant
\[  DH^*(X;A)/p=N^1H^*(X;A)/(p,\tilde N^1H^*(X;A))
 \]
 for the smooth projective approximation $X$ of $BG\times \bP^{\infty}$.
Here   $H^*(X;A)$ is the Betti (or \'etale) cohomology and $N^1H^*(X;A)$ (resp, $\tilde N^1H^*(X;A))$) is the coniveau (resp. $strong$ coniveau) filtration defined by
the kernel of the restriction maps to open sets  of $X$ (resp. the image of of Gysin maps).  For details see $\S 2$ below.

Hence $DH^*(X;A)/p$ is written as a sub-quotient module of $H^*(X;A)/p$.

Here an approximation (for $degree \le N$) is 
the projective (smooth)  variety $X=X_G(N)$ such that
there is a map $g:X\to BG\times \bP^{\infty}$ with \[g^*:H^*(BG\times \bP^{\infty};A)\cong H^*(X;A) \quad 
for \ *<N. \]
 (In this paper, we say $X$ is an approximation $for$ $BG$ when it is that $of$ $BG\times \bP^{\infty} $
strictly speaking.)    Let us write $DH^*(X;\bZ)$ by $DH^*(X)$ simply as usual.

 For example,  let $G=G_n$ be the
elementary abelian $p$-group $(\bZ/p)^n$.
Recall the $mod(p)$ cohomology  (for $p$ odd)
\[H^{*}(BG_n;\bZ/p) \cong \bZ/p[y_1,....y_n]\otimes
\Lambda(x_1,...,x_n)\]
where $|x_i|=1$ and $Q_0(x_i)=y_i$ for the Bokstein
operation  $Q_0=\beta$.  (Here $\Lambda(a,b,...)$ is the exterior algebra generated by $a,b,...).$

\begin{thm}
For any prim $p$, 
take $G=G_n=(\bZ/p)^n$,  $n\ge 2$ and $\alpha_i=Q_0(x_1x_2...x_i)\in H^{n+1}(X_{G_n})$.
Then we have
\[  DH^*(X_{G_n})/p\supset \bZ/p\{\alpha_2,\alpha_3,...,\alpha_n\}
 \quad *\le n+1<N.\]
\end{thm}

Hence $X_{G_n}$ is not stable rational.  Moreover 
$X_{G_n}$ and $X_{G_{n'}}$ are not stable birational
equivalent when $n\not =n'$.

Next we consider the (connected) case $G=SO_n$ the special orthogonal group ($p=2$).  Its cohomology is
\[ H^*(BSO_{2m+1};\bZ/2)\cong \bZ/2[w_2,w_3,....w_{2m+1}],\]
with $Q_0w_{2m}=w_{2m+1}$  where $w_i$ is the Stiefel-Whitney class for the embedding $SO_n \to O_n$.
Hence we can identify $w_{2i+1}\in H^*(BG)$.

\begin{thm}  \cite{YaC} Let $X_n=X_n(N)$ be approximations for $BSO_n$  for $n\ge 3$ and $ 2^{m+2}<N$. 
Then we have
\[  DH^*(X_{2m+1})/2\supset  \bZ/2\{w_3,w_5,...,w_{2m+1}\}
\quad for \  all \  2m+1\le  *<N.\]
\end{thm}

We consider the cases $G$ is a simply connected simple 
group.   Let $G$ contain $p$-torsion.  Then 
we know $H^4(BG)\otimes \bZ_p\cong \bZ_{p}$, and write its generator by $w$.  Then we have 
\begin{lemma}  \cite{YaC}   Let $G$ be a simply connected  group such that $H^*(BG)$ has $p$-torsion.
Let $X=X(N)$ be an approximation for $BG$ for 
$N\ge 2p+3$. Then
 \[ DH^4(X)/p\supset \bZ/p\{w\}.\]
\end{lemma}

Next, we study the retract rationality of $X_G$ for the above groups $G$.
We consider the Zariski cohomology
$H_{Zar}^*(X,\mathcal{H}_{A}^{*'})$
where $\mathcal{H}_{A}^*$ is the  Zariski sheaf induced from the  presheaf  given by $U\mapsto H_{\acute{e}t}^*(U;A)$
for an open $U\subset X$.  
It is well known   when $X$ is complete and smooth, the unramified cohomology is written 
\[H_{ur}^*(X;\bZ/p)
\cong H^0_{Zar}(X;\mathcal{H}^*_{\bZ/p}),\]
 and it 
is an invariant of the retract rationality of $X$
(Proposition 3.1. 3.4 in [Me]).

By Totaro \cite{Ga-Me-Se},  the above cohomology is also isomorphic to the cohomological invariant (of $G$-torsors) i.e.
\[  H^0_{Zar}(BG;\mathcal{H}^*_{\bZ/p})
\cong Inv^*(G;\bZ/p).\]

Let $H^{*,*'}(X;\bZ/p)$ be the $mod(p)$ motivic cohomology
of $X$ so that 
\[H^{*,*}(X;\bZ/p)\cong H^*_{\acute{e}t}(X;\bZ/p)\ \  and \ \  H^{2*,*}(X;\bZ/p)\cong CH^*(X)/p.\]
Let $0\not=\tau\in H^{0,1}(Spec(k);\bZ/p)$.  Then $\tau$ defines the map 
\[\tau:H^{*,*'}(X;\bZ/p)\to
H^{*,*'+1}(X;\bZ/p) \]
such that the cycle map is written
\[ CH^*(X)/p\cong H^{2*,*}(X)/p
\stackrel{\times \tau^*}{\to} H^{2*,2*}(X;\bZ/p)\cong
H^{2*}(X;\bZ/p).\]

From Orlov-Vishik-Voevodsky [Or-Vi-Vo], ([Te-Ya] for $p:odd$,)
we have 
\begin{lemma}  (\cite{Or-Vi-Vo})  We have  the short exact sequence 
\[    0\to  H^{*,*}(X;\bZ/p)/(\tau) \to
H^0_{Zar}(X;\mathcal{H}^*_{\bZ/p})  
 \to  Ker(\tau| H^{*+1,*-1}(X;\bZ/p)) \to 0.\]
\end{lemma}
Here \quad $ H^{*,*}(X;\bZ/p)/(\tau)=H^{*,*}(X;\bZ/p)/(\tau
H^{*.*-1}(X.\bZ/p)) $
 \[\cong   H^{*}(X;\bZ/p)/N^1H^*(X;\bZ/p).\]   
 This cohomology is called $stable$ cohomology,  and studied  by 
Bogolomov [Bo]. [Te-Ya2]. 

For example, 
when $G=(\bZ/p)^n$, it is known
\[  Inv^*(G;\bZ/p)\cong \Lambda(x_1,...,x_n).  \]

\begin{thm}
Let $G=(\bZ/p)^n$ and  $X=X_{G}$.  Then 
for $b_i=x_1...x_i$
\[ H^{*}_{ur}(X;\bZ/p)\supset  H^{*,*}(X;\bZ/p)/(\tau) 
\supset \bZ/2\{1, b_2, b_3,...,b_n\}.\]
Hence each $X_n$ and $X_{n'}$ are not retract birational
equivalent
when $n\not=n'$.
\end{thm}

By Serre  [Ga-Me-Se], when $G=SO_{2m+1}$, it is known
\[Inv^*(G;\bZ/2)\cong \bZ/2\{1,w_2,...,w_{2m}\}.\]
\begin{thm}
Let $G=SO_{2m+1}$ and $X=X_{G}$.  Then 
\[  H^{*}_{ur}(X;\bZ/2)\supset  H^{*,*}(X;\bZ/2)/(\tau) 
\supset \bZ/2\{1,w_2,...,w_{2m}\}.\]
Hence each $X_m$ and $X_{m'}$ are not retract rational
equivalent
when $m\not=m'$.
\end{thm}

We also give examples  of nonzero elements of $Ker (\tau)$ in Lemma 1.4.

\begin{thm}
Let $G$ be a simply connected simple group and $X=X_G$.
Then there is the element $w\in H^4(X;\bZ/p)$ 
such that  
\[H^3_{ur}(X;\bZ/p)\twoheadrightarrow
  Ker(\tau|H^{4,2}(X;\bZ/p)) \supset \bZ/p\{w\}
. \]
Hence $X$ is not retract rational.
 \end{thm}

{\bf Remark.}
It is known $BSpin_n$ for $n\le 14$ are stable rational
[Ko], [Me], \cite{Re-Sc}.
Hence $BG=BSpin_n$ for $7\le n\le 14$ and its approximation
$X=X_G$ are different stable rational type.

At the last three sections, we will do quite different arguments
from the preceding sections,  for
 quadrics $X$ over $\bR$.
Let us write
\[DH^{*}(X;\bZ_p)=DH^{*}_{\acute{e}t}(X;\bZ_p)
\quad \]
where \ \  
$H^*_{\acute{e}t}(X;\bZ_p)= Lim_{\infty\gets s}H^*_{\acute{e}t}(X;\bZ/{p^s}) \cong  Lim_{\infty\gets s}H^{*,*}(X;\bZ/{p^s}).$

In this paper, the \'etale cohomology
(with the integral coefficients $\bZ_2(*)$ for even degrees)
means the motivic cohomology   ;
\[  H_{\acute{e}t}^{2*}(X;\bZ_2(*))\cong \begin{cases}
   H^{2*,2*}(X;\bZ_2)\quad for\ *=even\\
 H^{2*,2*+1}(X;\bZ_2)\quad for\ *=odd.\end{cases}.
\]

Here 
we see the examples that
$X$ are not retract  rational ($H^{4*}_{ur}(X;\bZ_2)\not=\bZ/2$))  while  $DH^{2*}(X;\bZ_2(*))=0$.
 Let   $X=Q^d$ be the anisotropic quadric of dimension  $d=2^n-1$ (i.e. the norm variety).  
Then there are elements
\[ h\in H^2_{\acute{e}t}(X;\bZ_2(1))\quad and  \quad 
\bar \rho_4\in H_{\acute{e}t}^{4}(X;\bZ_2(0)).\]

\begin{thm}  ([Ya6])  
The ring $H_{\acute{e}t}^{2*}(Q^{2^n-1};\bZ_2(*))$ is multiplicatively generated by $\bar \rho_4$ and $h$ the hyper plane section.
\end{thm}

\begin{thm}
Let $X_n=Q^{2^n-1}$, $n\ge  2$ the norm variety.  Then
\[ DH^{2*}(X_n;\bZ_2(*))=0, \]
\[H_{ur}^{2*}(X_n;\bZ_2(*))\supset 
 \bZ/2[\bar \rho_4]/(
\bar\rho_4^{2^{n-1} }).\]
Hence for $n\not =n'$, we see that $X_n$ and $X_{n'}$ are not retract birational equivalent. \end{thm}


{\bf Remark.}
If $\bar \rho_4 \in \tilde N^1H^{2*}(X;\bZ_2(*))$,
the above theorem was just corollary of the 
Frobenius reciprocity (Lemma 2.2).  But it does not hold (moreover , 
we see $\bar \rho_4\not \in N^1H^{2*}(X;\bZ_2(*))$). 

\section{two coniveau filtrations}

Let us recall the coniveau filtration
of the cohomology with coefficients in  $A$ for $A=\bZ,\bZ_p,$
or $\bZ/p$,
\[ N^cH^i(X;A)=\sum_{Z\subset X} ker(j^*:H^i(X;A)\to
H^i(X-Z,A))\]
where $Z\subset X$ runs through the closed subvarieties of codimension at least $c$ of $X$,
and $j:X-Z\subset X$ is the complementary open immersion.

Similarly, we can define the $strong$ coniveau filtration 
by 
\[ \tilde N^cH^i(X;A)=\sum_{f:Y\to X} im(f_*:H^{i-2r}(Y;A)\to
H^i(X,A))\]
where the sum is over all proper morphism $f: Y\to X$ from a smooth complex variety $Y$ of $dim(Y)=dim(X)-r$ with $r\ge c$,
and $f_*$ its transfer (Gysin map).
It is immediate that $\tilde N^cH^*(X;A)\subset 
N^cH^*(X;A)$.

It is known that 
 when $X$ is proper, 
 $ \tilde N^cH^i(X;\bQ)=N^cH^i(X;\bQ)$ by Deligne.  However 
Benoist and Ottem (\cite{Be-Ot}) recently show that the above two coniveau filtrations are not equal
for $A=\bZ$.

Let $G$ be an algebraic group such that $H^*(BG;\bZ)$ has
$p$-torsion for the classifying space $BG$ is 
defined by Totaro \cite{To}, and Bogomolov \cite{Bo}. 
Then let us say that  an (Ekedahl) $ approximation$ for $BG$  (for $degree \le N$) is 
the projective (smooth)  variety $X=X_G(N)$ such that
there is a map $g:X\to BG\times \bP^{\infty}$ with \[g^*:H^*(BG\times \bP^{\infty};A)\cong H^*(X;A) \quad 
for \ *<N. \]

In the paper \cite{YaC},  we try to 
 compute the stable birational
invariant of $X$ (Proposition 2.4 in \cite{Be-Ot})
\[ DH^*(X;A)=N^1H^*(X;A)/(\tilde N^1H^*(X;A)) \]
for  projective approximations $X$ for $BG$ (\cite{Ek},\cite{To},\cite{Pi-Ya}).

Here we recall the 
Bloch-Ogus \cite{Bl-Og} spectral sequence such that its $E_2$-term is given by
\[E(c)_2^{c,*-c}\cong 
H_{Zar}^c(X,\mathcal{H}_{A}^{*-c})\Longrightarrow H_{\acute{e}t}^*(X;A)\]
where $\mathcal{H}_{A}^*$ is the  Zariski sheaf induced from the  presheaf  given by $U\mapsto H_{\acute{e}t}^*(U;A)$
for an open $U\subset X$.

The filtration for this spectral sequence is defined as the coniveau filtration 
  \[N^cH_{\acute{e}t}^*(X;A)= F(c)^{c,*-c}\]
where the infinite term $ E(c)^{c,*-c}_{\infty}\cong  F(c)^{c,*-c}/F(c)^{c+1,*-c-1}$.

Here we recall the motivic cohomology 
$H^{*,*'}(X;\bZ/p)$ defined by Voevodsky and Suslin (\cite{Vo1},\cite{Vo3},\cite{Vo4})
so that
\[ H^{i,i}(X;\bZ/p)\cong H^i_{\acute{e}t}(X;\bZ/p)\cong H^i(X;\bZ/p).\]

Let us write  
$H^*_{\acute{e}t}(X;\bZ )$ simply 
by $H^*_{\acute{e}t}(X)$ as usual.
Note that  $H^*_{\acute{e}t}(X)\not \cong H^*(X)$ in general,
while we have the natural map $H_{\acute{e}t}^*(X)\to H^*(X)$.

Let $0\not =\tau \in H^{0,1}(Spec(\bC);\bZ/p)$.  Then by the  multiplying $\tau$, we can define a map
$H^{*,*'}(X;\bZ/p)\to H^{*,*'+1}(X;\bZ/p)$.
By Deligne ( foot note (1) in Remark 6.4 in \cite{Bl-Og}) and
Paranjape (Corollary 4.4 in \cite{Pa}), it is proven that
there is an isomorphism of the coniveau spectral sequence
with the $\tau$-Bockstein spectral sequence $E(\tau)_r^{*,*'}$
(see also \cite{Te-Ya2}, \cite{Ya1}).

\begin{lemma} (Deligne)
Let $A=\bZ/p$.  Then we have 
the isomorphism of spectral sequence 
$E(c)_r^{c,*-c}\cong E(\tau)_{r-1}^{*,*-c} \quad 
for \ r\ge 2.$ 
Hence the filtrations are the same, i.e. 
$N^cH_{\acute{e}t}^*(X;\bZ/p)= F_{\tau}^{*,*-c}=Im(\times \tau^c:H^{*,*-c}(X;\bZ/p)) $. Thus 
we have the isomorphism
\[ H^{*,*}(X;\bZ/p)/(\tau)
\cong   H^{*}(X;\bZ/p)/N^1H^*(X;\bZ/p).\] 
\end{lemma}

We recall here the Frobenius reciprocity law.

\begin{lemma} (reciprocity law)
If $a\in \tilde N^*H^{2*}(X;A)$, then for each $g\in H^{*'}(X;A)$ we have
$ag\in \tilde N^*H^{2*+*'}(X;A)$.
\end{lemma}
\begin{proof}
Suppose we have $f:Y\to X$ with $f_*(a')=a$.
Then \[  f_*(a'f^*(g))=f_*(a')g=ag\]
 by the Frobenius reciprocity law.
\end{proof} 

Let $G$ be an algebraic group (over $\bC$) and $r$
be  a complex representation $r:G\to U_n$ the unitary group.  Then  we can define
the Chern class in $H^*(BG)$ by $c_i=r^*c_i^U$. Here the Chern classes
$c_i^U$  
in  $H^*(BU_n)\cong \bZ[c_1^U,...,c_n^U]$ ([Qu1]) are defined by
using  the Gysin map as $c_n^U=i_{n,*}(1)$ for

\[ i_{n,*} : H^*(BU_n)\cong H_{U_n}^*(pt.)\stackrel{i_{n,*}}{\to} H_{U_n}^{*+2i}(\bC^{\times i})\cong H^{*+2i}(BU_n)\]
where $H_{U_n}(X)=H^*(EU_n\times_{U_n}X)$ is the $U_n$-equivaliant cohomology.

Let us write  by $Ch^*(X;A)$ the Chern subring
which is the subring of $H^*(X;A)$ multiplicatively generated by all Chern classes.
\begin{lemma}
We have a quotient map
\[ N^1H^*(X;A)/(IdealCh^*(X;A))\twoheadrightarrow DH^*(X;A).\]
\end{lemma}

The following lemma is proved by 
 Colliot Th\'er\`ene and Voisin
\cite{Co-Vo} by using the
affirmative answer of the Bloch-Kato conjecture
by Voevodsky. (\cite{Vo3}. \cite{Vo4})
\begin{lemma}  (\cite{Co-Vo})  Let $X$ be a smooth complex variety.
Then any torsion element in $H^*(X)$ is in
$N^1H^*(X)$.
\end{lemma}

 \section{the main lemmas}

The Milnor operation $Q_n$  (in $H^*(-;\bZ/p)$)
is defined by $Q_0=\beta$ and for $n\ge 1$
\[ Q_n=P^{\Delta_n} \beta-\beta P^{\Delta_n}, \quad \Delta_n=(0,..,0,\stackrel{n}{1},0,...). \]
(For details see \cite{Mi}, \S 3.1 in \cite{Vo1})
where $\beta$ is the Bockstein operation and $P^{\alpha}$
for $\alpha=(\alpha_1,\alpha_2,...)$ is the fundamental base of the module of finite sums of products of reduced powers.
(For example $P^{\Delta_i}(y)=y^{p^i}$ for $|y|=2$.
and  $Q_n$ is a derivative.)
\begin{lemma} Let $f_*$ be the transfer (Gysin) map (for 
 proper smooth) $f:X\to Y$.
Then  $Q_nf_*(x)=f_*Q_n(x)$ for $x\in H^*(X;\bZ/p)$.
\end{lemma}

\begin{proof}The above lemma is known (see the proof of Lemma 7.1 in \cite{YaH}). The transfer $f_*$ is 
expressed as $g^*f_*'$ such that
\[ f_*'(x)=i^*(Th(1)\cdot x), \quad x\in H^*(X;\bZ/p)\]
for some maps $g,f', i$ and the Thom class $Th(1)$.  Since 
$Q_n(Th(1))=0$ and $Q_i$ is a derivation, we get the lemma.
\end{proof}

By Voevodsky \cite{Vo1}, \cite{Vo2}, we have  the $Q_i$ operation also in the 
motivic cohomology $H^{*,*'}(X;\bZ/p)$ with
$deg(Q_i)=(2p^i-1,p-1).$

\begin{lemma} We see that $Im(cl)^+\subset N^1H^{2*}(X;A)$.
\end{lemma}
\begin{proof}   From Lemma 2.1, we see
$H^{*.*'}(X;A)\subset  N^{*-*'}H^*(X;A)$.  We have $H^{2*,*}(X;A)\cong  CH^*(X)\otimes A.$
Since $2*>*$ for $*\ge 1$, we see $cl(y)\in N^{1}H^{2*}(X;A)$.
\end{proof}

Each element $y\in CH^*(X)\otimes A$ is represented by closed algebraic set 
supported $Y$, while $Y$ may be singular.
On the other hand,  by Totaro \cite{To}, we have the modified cycle map $\bar cl$ such that the usual cycle map is 
\[ cl:CH^*(X)\otimes A\stackrel{\bar cl}{\to} MU^{2*}(X)
\otimes_{MU^*}A\stackrel{\rho}{\to} H^{2*}(X;A)\]
for the complex cobordism theory $MU^*(X)$.
It is known \cite{Qu1} that elements in $MU^{2*}(X)$ can be 
represented by proper maps to X from   stable almost complex manifolds $Y$.
(The manifold $Y$ is not necessarily a 
 complex manifold.)

The following lemma is well known.
\begin{lemma}
If $x\in Im(\rho)$ for $\rho: MU^*(X)/p\to H^*(X;\bZ/p)$, then 
we have $Q_i(x)=0$ for all $i\ge 0$. 
\end{lemma}
\begin{proof}  
Recall  the 
connective  
Morava K-theory $k(i)^*(X)$ with $k(i)^*=\bZ/p[v_i]$, $|v_i|=-2p^i+2$, which has natural maps 
\[ \rho:MU^*(X)/p \stackrel{\rho_1}{\to} 
 k(i)^*(X)\stackrel{\rho_2}{\to} H^*(X:\bZ/p).\]

It is known that there is an exact sequence (Sullivan exact sequence) such that
\[  ...\stackrel{\rho_2}{\to}  H^*(X;\bZ/p)\stackrel{\delta}{\to}
k(i)^*(X) \stackrel{v_i}{\to}
 k(i)^*(X)\stackrel{\rho_2}{\to} H^*(X:\bZ/p)\stackrel{\delta}{\to}...\]
with $\rho_2\delta=Q_i$. Hence 
$Q_i\rho_2(x)=\rho_2\delta \rho_2=0.$
which 
implies $Q_i\rho(x)=0$.
\end{proof}

The following lemma is the $Q_i$-version of  one of results by Benoist and Ottem.
\begin{lemma}
Let $\alpha\in N^1H^s(X)$ for $s=3$ or $4$.
If $Q_i(\alpha)\not =0\in H^*(X;\bZ/p)$ for some $i\ge 1$,  then
\[ DH^s(X)/p\supset \bZ/p\{\alpha\},\quad 
DH^s(X;\bZ/p^t)/p\supset \bZ/p\{\alpha\}\ \ for\ t\ge 2.
\]
\end{lemma}
 \begin{proof}
Suppose $\alpha\in \tilde N^1H^s(X)$ for $s=3$ or $4$, i.e.
there is  a smooth $Y$ with $f:Y\to X$ such that
the transfer $f_*(\alpha')=\alpha$ for $\alpha'\in H^*(Y)$.

Then for $s=4, $
\[ Q_i(\alpha')=(P^{\Delta_i}\beta-\beta P^{\Delta_i})(\alpha')=(-\beta P^{\Delta_i})(\alpha') 
 = -\beta (\alpha')^{p^i}\]
\[ =-p^i(\beta \alpha')(\alpha')^{p^i-1}=0 \quad (by\
the \ Cartan\ formula)\]
since $\beta(\alpha')=0$ and $P^{\Delta_i}(y)=y^{p^i}$ for $deg(y)=2$.
(For $s=3$, we get also $Q_i(\alpha')=0$ since  $ P^{\Delta_i}(x)=0$ for $deg(x)=1$.)
This contradicts to the commutativity of $Q_i$ and $f_*$.
 
The case $A=\bZ/p^t$, $t\ge 2$ is proved similarly,  since 
for $\alpha'\in H^*(X;A)$ we see $\beta \alpha'=0\in H^*(X;\bZ/p)$. Thus we have this lemma.  \end{proof}

We  will extend the above Lemma  3.4 to $s>4$, by using 
$MU$-theory of Eilenberg-MacLane spaces.
Recall that $K=K(\bZ,n)$ is the Eilenberg-MacLane space
such that the homotopy group $[X,K]\cong H^n(X;\bZ)$, i.e., each element 
$x\in H^n(X;\bZ)$ is represented by a homotopy map 
$x: X\to K$. Let $\eta_n\in H^n(K;\bZ)$ 
corresponding the identity map.
We know the image $\rho(MU^*(K))\subset H^*(K;\bZ)/p$
by Tamanoi.

\begin{lemma} (\cite{Ta}, \cite{Ra-Wi-Ya})
Let $K=K(\bZ,n)$ 
We have the isomorphism 
\[  \rho :  MU^*(K)\otimes _{MU^*}\bZ/{p}
\cong\bZ/{p}[Q_{i_1}...Q_{i_{n-2}}\eta_n| 0<i_1<...<i_{n-2}] \]
where the notation $\bZ/p[a,...]$ exactly means $\bZ/p[a,...]/(a^2|\ |a|=odd)$. 
\end{lemma}

The following lemma is an extension of Lemma 3.4 for  $s> 4$.  (Here we use  $MU^*$-theory, and we assume $H^*(-;A)$ is the $Betti$ cohomology.) 
\begin{lemma}
Suppose that $H^*(X;A)$ is the Betti  cohomology.
Let $\alpha\in N^cH^{n+2c}(X)$, $n\ge 2$, $c\ge 1$. 
Suppose that  there is a sequence $0<i_1<...<i_{n-1}$ with 
\[ Q_{i_1}...Q_{i_{n-1}}\alpha \not =0 \quad   in\ H^*(X;\bZ/p).\]
Then $D^cH^{*}(X)/p=N^cH^{*}(X)/(p,\tilde N^cH^{*}(X)) \supset \bZ/p\{\alpha\}$.
\end{lemma}
\begin{proof}
Suppose $\alpha\in \tilde N^cH^{n+2c}(X)$,  i.e.
there is  a smooth $Y$ of $dim(Y)=dim(X)-c$ with $f:Y\to X$ 
such that 
the transfer $f_*(\alpha')=\alpha$ for $\alpha'\in H^n(Y)$.


Let  $ r: H^*(X)\to H^*(X;\bZ/p)$ be the reduction map.
We consider the commutative diagram for $I=(i_1,...,i_{n-2})$ and
$j=i_{n-1}$ 
\[ \begin{CD}
\alpha'\in H^{n}(Y)    @>{f_*}>> 
\alpha\in H^{n+2c}(X)  
\\
@V{Q_Ir}VV    @V{Q_Ir}VV \\
Q_I(\alpha')\in Im(\rho| MU^*(Y)) @>{f_*}>>   H^*(X;\bZ/p)  \\
@V{Q_j}VV    @V{Q_j}VV  \\
 0=  Q_{i_{n-1}}Q_I(\alpha')\in H^*(Y;\bZ/p)    @>{f_*}>>  
Q_{i_{n-1}}Q_I(\alpha)\in H^*(X;\bZ/p). 
\end{CD}
 \]

Identify  the map $\alpha': Y\to K$ with $\alpha'=(\alpha')^*\eta_n.$
We still see from Lemma 3.5,
   \[Q_I(\alpha')= Q_{i_1}...Q_{i_{n-2}}((\alpha')^*\eta_n)\in 
Im(\rho : MU^*(Y)\to H^*(Y;\bZ/p)).\]
From  Lemma 3.3,  we see
\[Q_{i_{n-1}}Q_I(\alpha')=Q_{i_{n-1}}Q_{i_1}...Q_{i_{n-2}}(\alpha')=0 \in H^*(Y;\bZ/p).\]  
Therefore $Q_{i_{n-1}}Q_I(\alpha)$
must be zero by the commutativity of $f_*$ and $Q_i$.
\end{proof}


\section{abelian $p$-groups}

At first, we assume $H^*(X)$ is the $ Betti$  cohomology so that 
the main lemma (Lemma 3.6)  holds.
However we will see the most  $irrational$ results
hold 
for each $k\subset \bC$.

From the main lemma, we have  

\begin{lemma}
Let $\alpha \in N^1H^{n+2}(X)$ and $Q_I(\alpha)\not =0\in H^*(X;\bZ/p)$
for some $I=(0<i_1<...<i_{n-1})$.  Let $X'$ be a smooth projective variety.
Then 
\[ DH^*(X\times X')/p \supset \bZ/p\{\alpha\otimes 1\}.\]Hence $X\times X'$ is not stable rational,
\end{lemma}
\begin{proof} 
The (Betti) cohomology  $H^*(X;\bZ/p)$ satisfies the Kunneth formula.  Hence we have 
\[ Q_I(\alpha \otimes 1)=Q_I(\alpha)\otimes 1\not =0\ \ in\ \  \sum_{s=0} H^{*-s}(X;\bZ/p)\otimes H^s(X';\bZ/p).\]
From the main lemma, we have the lemma.
\end{proof}

Let $G_n=\bZ/p^n$.  Recall the $mod(p)$ cohomology
(for $p$ odd)
\[H^{*}(BG_n;\bZ/p) \cong \bZ/p[y_1,....y_n]\otimes
\Lambda(x_1,...,x_n)\]
where $|x_i|=1$ and $Q_0(x_i)=y_i$, (for $p=2$, $x_i^2=y_i$).

\begin{cor} For $n\ge 3$, let $G_n=(\bZ/p)^n$.
Then $X_{G_n}$ is not stable rational.
Moreover $X_{G_n}$ and $X_{G_{n'}}$ are not stable birational equivalent  for $n\not=n'$.
\end{cor}
\begin{proof}
Take $G=G_3$ and $\alpha=Q_0(x_1x_2x_3)\in H^4(X_G)$.
The last statement follows from
$1\otimes...\otimes 1 \otimes \stackrel{s}{\alpha} \otimes 1\otimes ...\otimes 1
\not =0\in DH^*(X_{G_n}) .$
\end{proof}
We can take also another $\alpha$ for the proof of the last statement in the above corollary.
\begin{lemma}
Take $G=G_n=(\bZ/p)^n$ and $\alpha_i=Q_0(x_1x_2...x_i)\in H^{n+1}(X_G)$.
Then we have
\[ DH^*(X_{G_n})\supset \bZ/p\{\alpha_2,\alpha_3,...,\alpha_n\} .\]
\end{lemma}

Since $\alpha_n=0$ in $H^*(X_{G_{n-1}})$ we also see 
that $X_{n}$ and $X_{n-1}$ are not stable birational equivalence.

The more detailed expression of $DH^*(X)/p$ seems
somewhat complicated.  
\begin{thm}   
Let $G=(\bZ/p)^n$.  Then we have (for fixed large $N$)
\[ DH^{s+1}(X)/p
\cong
 \bZ/p\{Q_0(x_{i_1}...x_{i_s})|1\le i_1<...,<i_s\le n\}. \]
\end{thm}
\begin{proof}
The integral cohomology 
(modulo $p$) 
is isomorphic to
\[ H^*(BG)/p\cong Ker(Q_0)\]
\[\cong H(H^*(BG;\bZ/p);Q_0)
\oplus
Im(Q_0)\]
where $H(-;Q_0)=Ker(Q_0)/Im(Q_0)$ is the homology with the differential $Q_0$.
It is immediate that  $H(H^*(B\bZ/p;\bZ/p); Q_0)\cong \bZ/p$.
By the K$\ddot{u}$nneth formula, we have 
$H(H^*((BG;\bZ/p);Q_0)\cong ( \bZ/p)^{n\otimes}\cong \bZ/p$.
Hence we have 
\[ H^*(BG)/p\cong
 \bZ/p\{1\}\oplus Im(Q_0)\]
\[ \cong \oplus _s\bZ/p[y_1,...y_n](1,Q_0(x_{i_1}... .x_{i_s})| 1\le i_1<...< i_s\le n) \]
where the notation $R(a,...,b)$ (resp. $R\{a,..., b\}$)
means the $R$-submodule (resp. the free $R$-module)
generated by $a,...,b$. Here we note $H^+(BG)$ is just $p$-torsion.

Also note that $y_1,...,y_n$ are represented by the Chern classes $c_1$.  From Lemma 2.3, we see
$  Ideal(y_1,...,y_n)=0\in DH^*(X).$

We know $Q_i(x_j)=y_j^{p^i}$ and $Q_j$ is a derivation. 
  We have the theorem from Lemma 4.3 and 
the reciprocity  law
 \[ Q_{i_1}...Q_{i_{s-2}}Q_0(x_{i_1}...x_{i_s})
=y_{i_1}^{p^{i_1}}...y_{i_{s-2}}^ {p^{i_{s-2}}}
y_{i_{s-1}}x_{i_s}+...\not =0.\]
(Note the $n=|\alpha'|$ in Lemma 4.3 is written by $s-1$ here.)
\end{proof}

\begin{cor}  If $n\not =n'\ge 3$, then $X(N)_n$ and $X(N)_{n'}$
are not stable birational equivalent.
\end{cor}

The above corollary also holds when $ch(k)=0$ and $k$ is an algebraic closed field by the base change theorem.

For each field $k=\bar k$, it is known from Voevodsky (for $p; odd$)
\[H^{*,*'}(BG_n;\bZ/p) \cong \bZ/p[y_1,....y_n,\tau]\otimes
\Lambda(x_1,...,x_n)\]
where $deg(x_i)=(1,1)$ and $Q_0(x_i)=y_i$.
Therefor we can identify 
\[ Q_0(x_1...x_m) \in H^{*}_{\acute{e}t}(BG_n;\bZ/p)\quad when \ \bar k=k.\]

Let us write $H_{\acute{e}t}^*(X;\bZ_p)$ simply by 
$H_{\acute{e}t}^*(X)$.
Let $G$ be an algebraic group which has an approximation
$X_G$ such that 
\[ H_{\acute{e}t}^*(X_G;\bZ_p)\cong H^*(BG\times \bP^{\infty})\otimes \bZ_p
\quad for\ *<N\]

We consider the maps
\[ \psi\ :\ N^1H^*_{\acute{e}t}(X)\subset  H^*_{\acute{e}t}(X)\to H^*_{\acute{e}t}(\bar X)
\to \ H^*_{\acute{e}t}(X(\bC)) \to H^*(X(\bC)).
\]
\begin{lemma}
Let $k\subset \bC$ (not assumed an algebraic closed  field).
Let $\alpha\in N^1H^*_{\acute{e}t}(X)$ and 
 $Q_I(\psi (\alpha))\not =0\in H^*(X(\bC);\bZ/p).$ Then
\[ DH^*_{\acute{e}t}(X;\bZ_p)/p\supset  \bZ/p\{\alpha\}.\]
Hence $X$ is not stable rational.
\end{lemma}
\begin{proof}
By the assumption, the main lemma implies
$  DH^*(X(\bC))\supset \bZ/p\{\psi \alpha\}.$

This implies a contradiction if 
$\bZ/p\{\alpha\}=0 $ in $ DH^*_{\acute{e}t}(X).$
Similarly, the stable rationality for $X$ imlplies
that for $X(\bC)$, which is a contradiction.
(Note here, we do $not$ assume of the stable birational
invariance for $DH_{\acute{e}t}^*(X)$.)
\end{proof}
For example,  Lemma 4.3 holds for all $k\subset \bC$.


\section{connective groups, $SO_n$}

Let $SO_{n}$ be the special orthogonal
group. Its mod(2) cohomology is
\[H^*(BSO_n;\bZ/2)\cong \bZ/2[w_2,...,w_n] \]
where $w_i$ is the Stiefel-Whitney class for $SO_n\subset 
O_n$. We know $Q_0w_{2m}=w_{2m+1}$.

\begin{thm} (\cite{YaC}) Let $X_n=X_n(N)$ be approximations for $BSO_n$  for $n\ge 3$.  Moreover, let $|Q_1...Q_{2m-1}(w_{2m+1})|<N$.
Then we have
\[  DH^*(X_{2m+1})\supset  \bZ/2\{w_3,w_5,...,w_{2m+1}\}
\quad for \  all \  2m+1\le  *<N.\]
\end{thm}

{\bf  Remark.} When $G=SO_3$, the inclusion in the above theorem is isomorphic.   However, when $G=SO_5$, we can not see whether 
$Q_0(w_2w_4)\in H^7(X)$ is zero or not in $DH^7(X)/2$.

Let $G=SO_5$.  Indeed, we can see the homology by $Q_0$ is given
\[ H(H^*(BG;\bZ/2);Q_0)\cong \bZ/2[c_2,c_4]\quad where \ c_i=w_i^2,\]
\[Im(Q_0)\cong \bZ/2[c_2,c_3,c_4,c_5](Q_0(w_2).Q_0(w_4),Q_0(w_2w_4)).\]
Hence $H^*(BG)/2$ is generated by $1,w_3,w_5. Q_0(w_2w_4)$ as a
$\bZ/2[c_2,c_3,c_4,c_5]$-module.  Hence we have 
\begin{lemma}  Let $G=SO_5$.
There ie a surjection 
\[ \bZ/2\{w_3.w_5, Q_0(w_2w_4)\} \twoheadrightarrow
DH^*(X_G)/2.\]
\end{lemma}

\begin{cor} Let $X_n=X_n(N)$ be approximation for $BSO_n$  for $n\ge 3$.  For $m\not =m'$, we see that $X_{2m+1}$ and $X_{2m'+1}$
are not stable rational equivalence. 
\end{cor}

The above corollary holds for all $k\subset \bC$, by
the similar arguments done in the last places in the preceding section.



\section{simply connected simple groups}

We next consider simply connected groups. Let us write by $X$ an 
approximation for $BG_2$ for the exceptional simple 
group $G_2$ of $rank=2$.
The $mod(2)$ cohomology is generated by the Stiefel-Whitney classes $w_i$ of
the real representation $G_2\to SO_7$
\[ H^*(BG_2;\bZ/2)\cong \bZ/2[w_4,w_6,w_7],
\quad P^1(w_4)=w_6,\ Q_0(w_6)=w_7, \]
\[H^*(BG_2)\cong (D'\oplus D'/2[w_7]^+)\quad where \ \ D'=\bZ[w_4,c_6].\]
Then we have  $Q_1w_4=w_7, Q_2(w_7)=w_7^2=c_7$ (the Chern class).

The Chow ring of $BG_2$ is also known
\[ CH^*(BG_2)\cong (D\{1,2w_4\} \oplus D/2[c_7]^+)\quad where \ \ D=\bZ[c_4,c_6]\ \ c_i=w_i^2.\]
In particular the cycle map $cl: CH^*(BG)\to H^*(BG)$ is injective.

It is known \cite{YaC} that $w_4\in N^1H^*(X;\bZ/2)$ and moreover  we can identify $w_4\in N^1H^*(X)$.
Since $Q_1(w_4)=w_7\not =0$,  from Lemma 4.1,
we have
$DH^4(X)\not =0$.  This fact is also written 
 in  \cite{Be-Ot}.  Moreover the isomorphism 
 $H^*(BG)/(c_4,c_6,c_7)\cong 
\Lambda(w_4,w_7)$ implies 
\begin{prop} (\cite{YaC})    For  $X$ an approximation for $BG_2$,
we have the surjection
\[ \Lambda(w_4,w_7)^+
\twoheadrightarrow 
DH^*(X)/2\quad for\ all \  *<N.\]
\end{prop}

{\bf Remark.} We can not see $w_7, w_4w_7 =0 $ or nonzero
in $DH^*(X)/2$.

The  cohomology of  other simply connected simple groups (with $2$-torsion) are written for example 
\[ H^*(BSpin_7;\bZ/2)\cong H^*(BG_2;\bZ/2)\otimes \bZ/2[w_8],\quad \]
\[ H^*(BSpin_8;\bZ/2)\cong H^*(BG_2;\bZ/2)\otimes \bZ/2[w_8,w_8'],... \] 
For the above groups $G$, there are the map $j: G_2\to G$ and the non zero element $w\in H^*(G)$ such that $j^*w=w_4$.

\begin{prop}  (\cite{YaC})    Let $G$ be a simply connected group such that $H^*(BG)$ has $p$-torsion.
Let $X=X(N)$ be an approximation for $BG$ for 
$N\ge 2p+3$. 
Then 
 there is 
$w\in H^4(X)$ such that 
      \[ DH^4(X)/p\supset 
\bZ/p\{w\} \]
 Hence these $X$ are not stable rational.
\end{prop}
$Proof$.
It is only need to prove the theorem when  $G$ is a simple group
having $p$ torsion in $H^*(BG)$.
Let $p=2$.  
It is well known that there is an embedding  $j:G_2\subset G$ such that  (see \cite{Pi-Ya}, \cite{YaRo} for details)
\[ H^4(BG)\stackrel{j^*}{\cong} H^4(BG_2)\cong \bZ\{w_4\}.\] 

Let $w=(j^{*})^{-1}w_4\in H^4(BG)$.  
From Lemma 3.1 in \cite{YaRo}, we see that $2w$ is represented by Chern classes.
Hence  $2w$ is the image from $CH^*(X)$,  and so 
$2w\in N^1H^4(X)$.  This means there is an open set $U\subset X$ such that
$ 2w=0\in H^*(U)$ that is, $w$ is $2$-torsion in $H^*(U)$. Hence from Lemma 2.4,
we have $w\in N^1H^4(U)$, and so there is $U'\subset U$ such that $w=0\in H^4(U')$.
This implies $w\in N^1H^4(X)$.

Since $j^*(Q_1x)=Q_1w_4=w_7$, we see $Q_1w\not =0$.
From the main lemma (Lemma 4.1), we see $DH^4(X)\not =0$
for $G$.

For the cases $p=3,5$, we consider the exceptional groups $F_4,E_8$ respectively.  Each simply connected simple group
$G$ contains $F_4$ for $p=3$, $E_8$ for $p=5$. There is $w\in H^4(BG)$ such  that $px$ is a Chern class 
\cite{YaRo}, and $Q_1w)\not =0\in H^*(BG;\bZ/p).$
In fact,  there is embedding $j:(\bZ/p)^3\subset G$ with $j^*(w)=Q_0(x_1x_2x_3)$. 
Hence we have the theorem.
\begin{cor}  Let $X$ be an approximation for $BSpin_n$
with $n\ge 7$ or $BG$ for an exceptional group $G$.  Then $X$ is not stable rational.
\end{cor}

\section{retract birational and 
unramified cohomology}

Here we 
note the relations to retract rationally.
Recall (in $\S 2$) that 
Bloch-Ogus
give a spectral sequence such that its $E_2$-term is given by
\[E(c)_2^{c,*-c}\cong 
H_{Zar}^c(X,\mathcal{H}_{A}^{*-c})\Longrightarrow H_{\acute{e}t}^*(X;A).\]
  By Orlov-Vishik-Voevodsky [Or-Vi-Vo], ([Te-Ya2] for $p:odd$,)
we know
 \begin{lemma} ([Or-Vi-Vo], [Vo5]) There is the long exact sequence 
  \[ H_{Zar}^{m-n-1}(X;\mathcal{H}_{\bZ/p}^n) \to H^{m,n-1}(X;\bZ/p)\stackrel{\times \tau}{\to}
  H^{m,n}(X;\bZ/p)\qquad \qquad  \]
  \[\qquad \qquad \to
H_{Zar}^{m-n}(X;\mathcal{H}_{\bZ/p}^n)\to H^{m+1,n-1}(X;\bZ/p)\stackrel{\times \tau}{\to}....\]
\end{lemma}
In particular, when $m=n$, the first $\times \tau$ is injective.

\begin{cor}    We have  the short exact sequence 
\[    0\to  H^{*,*}(X;\bZ/p)/(\tau) \to
H^0_{Zar}(X;\mathcal{H}^*_{\bZ/p})  \]
\[ \to  Ker(\tau: H^{*+1,*-1}(X;\bZ/p)
\to H^{*+1,*}(X;\bZ/p))\to 0.\]
\end{cor}
(Note 
$ H^{*,*}(X;\bZ/p)/(\tau)\cong H^{*}(X;\bZ/p)
/(N^1H^*(X;\bZ/p)).$  Hence we also write 
it as $H^*(X;\bZ/p)/N^1$.  This cohomology is called 
a stable cohomology and studied well by Bogomolov [Bo],
[Te-Ya2]

{\bf Remark.}  The $\bZ/2^s$ coffeciants version 
of Lemma 7.1, Corollary 7.2 also
hold.

The unramified cohomology is written by this 
$H^0_{Zar}(X;\mathcal{H}^*_{\bZ/p})$,
  when $X$ is complete,
\[H_{ur}^*(X;\bZ/p) = H_{ur}^*(k(X);\bZ/p) 
\cong H^0_{Zar}(X;\mathcal{H}^*_{\bZ/p}),\]
 and it 
is an invariant of the retract rationality of $X$
(Lemma 3.1, 3.4 [Me]). 

By Totaro [Ga-Me-Se], the cohomological invariant of $G$
is written (while $BG$ is not complete)
\[ Inv^*(G;\bZ/p)\cong H^0_{Zar}(BG;\mathcal{H}_{\bZ/p}^*),\]

Here we consider the following lemma which shows the relation among $DH^*(X_G)$, $Inv^*(G)$ and $H_{ur}^*(X_G)$.
\begin{lemma}
Assume that $0\not =x \in H^m(BG;\bZ/p)/(N^1)$
and $x$ is dedicated by $A_m=(\bZ/p)^m$ i.e. $res_{/N}(x)\not=0$ for the restriction (of stable cohomologies)
\[res_{/N}: H^*(BG;\bZ/p)/N^1 \to H^*(BA_m;\bZ/p)/N^1
\cong \Lambda(x_1,...,x_m).\]
Then (for projective approximation $X$ for $BG$) we have \[
\begin{cases}
 Inv^*(G;\bZ/p)\supset \bZ/p\{x\}, \\ 
 H^*_{un}(X;\bZ/p)\supset \bZ/p\{x\} , \\
 DH^*(X)/p\supset \bZ/p\{Q_0(x)\} .
\end{cases} \]
\end{lemma}
\begin{proof}
The first formula follows from
\[ Inv^*(G;\bZ/p)\cong H^0(BG;\mathcal{H}^*_{\bZ/p})
\supset H^*(BG;\bZ/p)/N^1.\]
The fact $x\not=0$ in $Inv^*(G;\bZ/p)$ follows from 
that $x$ is dedicated. 

The second formula comes from
$H_{ur}^*(X;\bZ/p)\cong  H^0(X;\mathcal{H}^*_{\bZ/p})$
where $X$ is smooth projective.

The last formula follows from the main lemma
(Lemma 3.4).
Let 
$Q_0(x)=\alpha \in 
\tilde N^cH^{n+2c}(X)$,  ($m=n+2c-1)$,  i.e.
there is  a smooth $Y$ of $dim(Y)=dim(X)-c$ with $f:Y\to X$ 
such that 
the transfer $f_*(\alpha')=\alpha$ for $\alpha'\in H^n(Y)$.

Identify  the map $\alpha': Y\to K$ with $\alpha'=(\alpha')^*\eta_n.$
We still see from Lemma 3.5,
   \[Q(\alpha')= Q_{i_1}...Q_{i_{n-2}}((\alpha')^*\eta_n)\in 
Im(MU^*(Y)\to H^*(Y;\bZ/p)).\]
From  Lemma 3.4,  we see
\[Q_{i_{n-1}}Q(\alpha')=Q_{i_{n-1}}Q_{i_1}...Q_{i_{n-2}}(\alpha')=0 \in H^*(Y;\bZ/p).\]  

Therefore $Q_{i_{n-1}}Q(\alpha)$
must be zero by the commutativity of $f_*$ and $Q_i$.
But $Q_{i_1}...Q_{i_{n-1}}Q_0(x)\not =0$ from  the assumption that $x$ is deduced from $A_{n+1}$.
In fact in $H^*(BA_{n+1};\bZ/p)$, we see (without $mod(N^1)$)
\[ Q_{i_1}...Q_{i_{n-1}}Q_0(x_1...x_{n+1})=
y_1^{p^{i_1}}...y_{n-1}^{p^{i_{n-1}}}
 y_n x_{n+1} +...\not =0.\]
\end{proof}

Now we consider the examples.
At first, we consider the case $G=A_n=(\bZ/p)^n$.
 and $X=X_{G}$.  It is known
from Garibarldy-Merkurjev-Serre  [Ga-Me-Se], Theorem 6.3 in [Te-Ya2] that \[
 Inv^*(G;\bZ/2)\cong H^{*,*}(BG;\bZ/2)/(\tau)
\cong \Lambda(x_1,...,x_n).\]
Since $X$ is (proper) approximation of $BG$, we have
\begin{thm}
Let $G=G_{n}=(\bZ/p)^n$ and $X=X_{G}$.  Then 
\[  H^{2*}_{ur}(X;\bZ/p)\supset  H^{2*,2*}(X;\bZ/p)/(\tau)
\cong  \Lambda(x_1,x_2,...,x_n)\]
in Corollary 7.2. 
\end{thm}
 Writing $\alpha_i=
Q_0(x_1...x_i)$,  we still  have (Lemma 4.3)
\[ DH^*(X)/p\supset \bZ/p\{ 
\alpha_2,\alpha_3,...,\alpha_n\}.
\]
Then $X_{G_n}$ and $X_{G_{n'}}$ are not retract rational equivalent 
if $n\not =n'$.

{\bf Remark.}
From  ( Saltman [Sa]) , it is well known that  there is a finite group $G$ (e.g. $|G|=p^7$, $p:odd)$) such that 
\[ 0\not = x_2\in H^2_{ur}(k(W)^G;\bZ/p) \cap  H^{2,2}(BG;\bZ/p)/(\tau)
\quad  .\]
Here $G$ acts freely on a $\bC$-vector space $W$, and we have
\[ H_{ur}(K(W)^G;\bZ/p)\cong H_{Zar}(W//G;\mathcal{H}^*_{\bZ/p})
\subset H_{Zar}(BG;\mathcal{H}_{\bZ/p}^*)\]
such that $k(W//G)\cong k(W)^G$.
Hence $H^*_{un}(k(W)^G;\bZ/p)\not \cong H^*(k;\bZ/p)$.
So $k(W)^G$  is not purely transcendent over $k$.
(Hence $BG$ is not retract rational.)

{\bf Remark.}
We do $not$  assume $H^0_{Zar}(X;\mathcal{H}_{\bZ/p}^*)\cong 
H^0_{Zar}(X';\mathcal{H}_{\bZ/p}^*)$ for an other approximation $X'$.


Next we consider the case $G=SO_{2m+1}$ and $X=X_{G_m}$.  It is known
from Garibarldy-Merkurjev-Serre  [Ga-Me-Se], Theorem 6.3 in [Te-Ya] that \[
 Inv^*(G;\bZ/2)\cong H^{*,*}(BG;\bZ/2)/(\tau)
\cong \bZ/2\{1,w_2,...,w_{2m}\}.\]
Since $X$ is (proper) approximation of $BG$, we have
\begin{thm}
Let $G=SO_{2m+1}$ and $X=X_{G}$.  Then 
\[  H^{2*}_{ur}(X;\bZ/p)\supset  H^{2*,2*}(X;\bZ/2)/(\tau)
\supset \bZ/2\{1,w_2,...,w_{2m}\}\]
in Lemma 7.1.  
\end{thm}
We also have (Theorem 5.2) \ \ 
$ DH^*(X)/2\supset \bZ/2\{Q_0(w_2),...,Q_0(w_{2m})\}.$
Hence $X_{G_m}$ and $X_{G_{m'}}$ are not retract rational
if $m\not =m'$.

From Theorem 5.2 and the preceding theorem, we have 
\begin{cor}  
Let $G_n'=SO_{n}$ and $X=X_{G}$.  Then
 $X_{G_n}$ and $X_{G_{n'}}$ are not retract rational
if $n\not =n'$.
\end{cor}
\begin{proof}
By Serre  \cite{Ga-Me-Se}, we know
\[ Inv^*(BSO_{2m};\bZ/2)\cong \bZ/2\{1,w_2,...,w_{2m-2},u_{2m-1}\}  \]
with $ |u_{2m-1}|=2m-1$.
We see $X_{2m}$ and $X_{2m+1}$ are not
retract rational since $w_{2m+1}$ is zero in the invariant for
$BSO_{2m}$.  
We see $X_{2m-1}$ and $X_{2m}$ are not
retract rational since $u_{2m}$ is zero in the invariant for
$BSO_{2m-1}$.
\end{proof}

  { \bf Remark.}    Kordonskii [Ko], Merkurjev (Corollary 5.8 in [Me]),
and Reichstein-Scavia show [Re-Sc]
that  $BSpin_n$ itself is stably rational when $n\le 14$.
 These facts imply that  
the (Ekedahl) approximation $X$  is not stable rationally equivalent to 
$BG$.
(The author thanks Federico Scavia who pointed out this remark.)

At last of this section, we consider the case $G=PGL_p$
projective general linear group.
We have (for example Theorem 1.5,1.7 in \cite{Ka-Ya})  
additively \[H^*(BG;\bZ/p)\cong 
M\oplus N \quad with \ \ 
 M\stackrel{add.}{\cong}\bZ/p[x_4,x_6,...,x_{2p}],\quad \]
\[ N= SD\otimes \Lambda(Q_0,Q_1)\{u_2\}\quad with\ \ 
SD=\bZ/p[x_{2p+2},x_{2p^2-2p}]\]
where $x_{2p+2}=Q_1Q_0u_2$ and suffix means its degree.  The Chow ring is given as 
\[ CH^*(BG)/p\cong M\oplus SD\{Q_0Q_1(u_2)\}.\]

From Lemma 7.3, we have :

\begin{thm}  Let $p$ be odd.
For an approximation $X$ for $BPGL_p$, we see 
\[ DH^*(X)/p\supset  \bZ/p\{Q_0u_2\} ,\]
\[ H_{un}^*(X;\bZ/p)\supset  \bZ/p\{1,u_2\},
\qquad Inv^*(G;\bZ/p)\supset  \bZ/p\{1,u_2\}.\]
\end{thm}

In the above case, we do  not see here that $DH^*(X)$ for $*<N$
is invariant
of $BG$, (under taking another $X'$  
as  approximations for $G$).

\section{Retract rational for simply connected $G$}

We will see that simply connected groups $G$ satisfy the  similar
facts, but   such as 
$Ker(\tau| H^{*+1,*-1}(X;\bZ/p))\not =0$  in  Lemma 7.1.
In $\S 6$, we see there is $0\not= w\in H^{4}(X)$
such that $DH^4(X)/p\supset \bZ/p\{w\}$.
We will see that this $w$ corresponds a nonzero element
in $ H_{un}^3(X;\bZ/p)$.
\begin{thm} (\cite{YaRo}) Let $G$ be a simply connected simple group.
Then there is the element (Rost invariant) such that  
\[H^3_{ur}(X;\bZ/p)\twoheadrightarrow
  Ker(\tau|H^{4,2}(X;\bZ/2)) \supset \bZ/p\{w\}
. \]
Hence $X$ is not retract rational.
\end{thm}

\begin{proof}

We consider the following diagram
\[ \begin{CD}
H_{ur}^*(BG;\bZ/p) @>j^*>> H^*_{ur}(X;\bZ/p) \\
@V(1)VV   @V{(2)\ \cong}VV  \\ 
    Inv^*(G:\bZ/p)\cong H^{0}(BG;\mathcal{H}_{\bZ/p}^*)
@>j^*>>
H^{0}(X; \mathcal{H}_{\bZ/p}^*)\twoheadrightarrow
 Ker(\tau)
.\end{CD}\]

Here $H_{un}^+(BG;\bZ/p)=0$
when $BG$ is retract rational.  
(The map (1) need not isomorphism.)
We see that 
the map 
$ (2) :H_{ur}^*(X;\bZ/p)
\cong 
H^{0}(X;\mathcal{H}^*_{\bZ/p})$ 
because $X$ is projective and smooth. 
Recall Lemma 7.1 that we have the surjection 
\[ H^0(X;\mathcal{H}^*_{\bZ/p})
\twoheadrightarrow
 Ker(\tau|H^{*+1,*-1}(X:\bZ/p)).\]

Hereafter, we consider the case $*=3$.
We consider the following commutative diagram.
\[ \begin{CD}
pw\in H^{4,4}(BG;\bZ/p^2) @>{j^*\ \cong}>>
pw\in H^{4,4}(X;\bZ/p^2)  @>>> 0\in H^{4.4}(X;\bZ/p)\\
@A{\tau'}AA@A{\tau'}AA  @A{\tau \ (inj.)}AA\\
c_2' \in H^{4,3}(BG;\bZ/p^2) @>{j^*}>>
c_2''\in H^{4,3}(X;\bZ/p^2) @>>> 0\in H^{4,3}(X;\bZ/p)  \\
@A{\tau' }AA@A{\tau'}AA @A{\tau}AA
\\ c_2'\in H^{4,2}(BG;\bZ/p^2) @>{j^*}>>
c_2''\in H^{4,2}(X;\bZ/p^2) @>{r}>> c'''\in H^{4.2}(X:\bZ/p)\\
\end{CD}
\]

From the proof in Proposition 6.2,  we see that there is
$ c_2'\in H^{4,2}(BG;\bZ/p^2)$
so that (for 
$\tau':
H^{*,*'}(X;\bZ/p^2)
\to H^{*,*'+1}(X;\bZ/p^2))$ we have 
\[ (\tau')^2c_2'=pw\in H^{4,4}(BG;\bZ/p^2).\]
(In fact $pw$ is represented by a Chern class,
but $w$ itself is not in the image of the cycle map.)

Next take $c''=j^*c_2'\in H^{4,3}(X;\bZ/p^2)$.  Since
$j$ is a projective approximation, we have 
\[H^{4,4}(BG;\bZ/p^2)
\cong H^{4,4}(X;\bZ/p^2).\]  Here 
$(\tau')^2c''=pw$.  Hence  $c''\not=0\in
H^{4,2}(X;\bZ/p^2)$.

Let us write by $c'''$  the image of $c''$ in  $ H^{4,2}(X;\bZ/p)$.  
We note $c'''\in Ker(\tau)|H^{4,2}(X;\bZ/p),$
because $\tau:H^{4,3}(X;\bZ/p)\to H^{4,4}(X;\bZ/p)$
is injective from [Or-Vi-Vo].

Moreover, $c'''$ is a module generator in 
$Ker(\tau)$, in fact if $c''=px$, then $\tau^2x=w$
which is not $Ker(\tau)$.

Hence there is $a\in H_{ur}^3(X;\bZ/p)$ which corresponds
$c'''\in Ker(\tau|H^{4,2}(X;\bZ/p).$
\end{proof}

\begin{cor} Let $G$ be a simply connected group 
having $p$-torsion in $H^*(BG)$,
and $X=X_G$ be a projective approximation for $BG$.
Then $H^3_{ur}(X;\bZ/p)\not =0$ and so $X$ is not
retract rational.
\end{cor}

In the last of this section,  we consider the case $G=F_4$,
$p=2$ 
the exceptional simple group of rank $4$.  By  \cite{Ga-Me-Se}, the cohomology invariant is known
\[ Inv^*(G;\bZ/2)\cong \bZ/2\{1,u_3,f_5\} \quad |u_3|=3,\ |f_5|=5.\]
Since $H^{5,5}(BG;\bZ/2)=0$,  we  know
$f_5$ corresponds
\[ 0\not =x\in Ker(\tau|H^{6,4}(BG;\bZ/2))\twoheadleftarrow
H^0(BG;\mathcal{H}_{\bZ/2}^5).\]
But we can $not$ say here that $0\not=x\in H^{6,4}(X;\bZ/2)$.
\begin{prop}
If there is an approximation such that $H^{6,4}(BG;\bZ/2)
\cong H^{6,4}(X;\bZ/2)$, then 
\[ H_{ur}^*(X;\bZ/2)\supset \bZ/2\{u_3,f_5\}.\]
Hence if  the assumption is correct. then $X_{G_2}$ and $X_{F_4}$ are not retract rational
equivalent.
\end{prop}

\section{extraspecial $p$-groups}

We assume at first that $p$ is an odd prime. 
The extraspecial $p$-group $E(n)=p_+^{1+2n}$ is the group such that exponent is $p$,
its center is $C\cong \bZ/p$ and there is the extension
\[ 0\to C\to E(n)\stackrel{\pi}{\to} V_n\to 0\]
\noindent  with   $V=\oplus ^{2n}\bZ/p$. 
(For details of the cohomology of $E(n)$ see [Te-Ya1].) 
We can take generators  $a_1,...,a_{2n},c\in E(n)$ such that $\pi(a_1),..
, \pi(a_{2n})$ (resp. $c$ ) make a base of $V_n$ (resp. $C$) such that commutators are 
\[  \quad [a_{2i-1},a_{2i}]=c  \quad and \quad [a_{2i-1},a_j]=1 
\ \ if\ j\not = 2i. \]
We note that $E(n)$ is also the central product of the $n$-copies of $E(1)$
\[E_n\cong E(1)\cdot \cdot \cdot E(1)=E(1)\times_{\la c\ra}E(1)...\times_{\la c\ra} E(1).\]

Take cohomologies
\[
H^*(BC;\bZ/p)\cong \bZ/p[u]\otimes \Lambda(z), \quad   \beta z=u,
\]
\[
H^*(BV_n;\bZ/p)\cong \bZ/p[y_1,...,y_{2n}]\otimes \Lambda (x_1,...x_{2n}),
  \quad \beta x_i=y_i,\]
identifying the dual of $a_i$ (resp.$c$) with $x_i$ (resp. 
$z$).  That means
\[ H^1(E(n);\bZ/p)\cong Hom(E(n);\bZ/p)\ni x_i :a_j\mapsto \delta_{ij}.\]
The central extension  is expressed by 
\[
f=\sum _{i=1}^nx_{2i-1}x_{2i}\in H^2(BV_n;\bZ/p).
\]
Hence  $\pi ^*f=0$ in $H^2(BE(n);\bZ/p)$.
We consider the Hochshild-Serre spectral sequence
\[E_2^{*,*'}\cong H^*(BV_n;\bZ/p)\otimes H^*(BC;\bZ/p)\Longrightarrow
H^*(BE(n);\bZ/p).\]
Hence the first nonzero differential is $d_2(z)=f$ and the next differential is
\[ d_3(u)=d_3(Q_0(z))=Q_0(f)
 =\sum y_{2i-1}x_{2i}-y_{2i}x_{2i-1}.\]
In particular 
\[E_4^{0,*}\cong \bZ/p[y_1,...,y_{2n}]\otimes \Lambda (x_1,...x_{2n})/(f,Q_0(f)).\]


\begin{lemma}  We have the inclusion
\[\Lambda(x_1,...,x_{2n})/(f)\subset H^*(BE(n);\bZ/p).\]
\end{lemma}
\begin{proof} We consider similar group $E(n)'$ such that
its center is $C\cong \bZ/p$ and there is the extension
\[ 
\begin{CD}
0 @>>> C@>i>> E(n)' @>{\pi}>> V_n' @>>> 0 \\
\end{CD}\]
but   $V_n'=\oplus ^{2n}\bZ_p$ such that there is the quotient map $q:E(n)'\to E(n)$.
We also consider the spectral sequence
\[E_2^{*,*'}\cong H^*(BV';\bZ/p)\otimes H^*(BC;\bZ/p)\Longrightarrow
H^*(BE(n)';\bZ/p).\] 
Here $H^*(BV_n';\bZ/p)\cong \Lambda(x_1,...x_{2n})$. The first nonzero differential
is $d_2(z)=f$ but the second differential is
\[d_3(u)=\sum y_{2i-1}x_{2i}-y_{2i}x_{2i-1}=0
\in \Lambda(x_1,...,x_{2n}).\]
Hence $E_4^{*,*'}$ is
(multiplicatively)  generated by $u$ and $x_i$ (permanent cycles).
So $E_4^{*,*'}\cong E_{\infty}^{*,*'}$.  Therefore we have
\[H^*(BE(n)';\bZ/p)\cong \bZ/p[u]\otimes\Lambda(x_1,...,x_{2n})/(f).\]
From the map $q^*: H^*(BE(n);\bZ/p)\to H^*(BE(n)';\bZ/p)$,
we get the result.
\end{proof}
However $H^*(BE(n);\bZ/p)/(N^1)\not \cong  \Lambda(x_1,...,x_{2n})/(f)$, in fact,
when $n=1$, from Theorem 3.3 in [Ya6]
we see
\begin{prop}  We have 
\[ H^*(BE(1) ;\bZ/p)/(N^1)
\cong \bZ/p\{1,x_1,x_2,a_1',a_2'\} \quad deg(a_i')=2.\]
\[ DH^*(X)/2\cong 
\bZ/2\{Q_0(a_1'),Q_0(a_2')\}.\]
\end{prop}

\begin{lemma}  Let $n\ge 2$.  Then 
\[y_i^py_j-y_iy_j^p\not =0 \quad \in \ H^*(BE(n);\bZ/p).\]
\end{lemma}
\begin{proof}
By the inclusion $E(2)\subset E(n)$ and induced quotient map
\[H^*(BE(n);\bZ/p) \to H^*(BE(2);\bZ/p)\]
 we only need to see
\[  y_1^py_2-y_1y_2^p\not =0 \quad \in H^*(BE(2);\bZ/p)\otimes \bar\bF_p   \]
for the algebraic closure $\bar \bF_p$ of the finite field $\bF_p$.

Let $n=2$. Note here 
\[ Q_iQ_0(f)=
y_1^{p^i}y_2-y_1y_2^{p^i}
+y_3^{p^i}y_4-y_3y_4^{p^i}  \]
\[  =  y_2\Pi_{\lambda\in \bF_{p^i}}(y_1-\lambda y_2)  
+y_4\Pi_{\lambda\in \bF_{p^i}}(y_3-\lambda y_4).\]

Hence this formula $Q_iQ_0(f)$ is a sum of
multiplies of 
\[ y_1^py_2-y_1y_2^p=y_2\Pi_{\lambda\in \bF_{p}}(y_1-\lambda y_2) 
 \quad and  \quad 
y_3^py_4-y_3y_4^p, \quad 
\]

Suppose that  $y_1^py_2-y_1y_2^p=0$.
Then by the symmetry of the group. we see
$y_3^py_4-y_3y_4^p=0$.  But  it is known 
\cite{Te-Ya1}  $(Q_1Q_0(f), Q_2Q_0(f))$ is regular in $\bZ/p[y_1,y_2,y_3.y_4] $.
This is a contradiction.
\end{proof}

The more concrete expression of $DH^*(X)/p$ seems
somewhat complicated.  So we only give it for $*=3$.
\begin{prop}   
Let $G=E(n),\ n>1$.  Then we have
\[ DH^3(X)/p
\cong
 \bZ/p\{Q_0(x_ix_j)|(i,j)\not =(1,2),\ 1\le i<j\le n\}. \]
\[ H^2_{ur}(X;\bZ/p),\ Inv^2(G:\bZ/p)\supset 
 \bZ/p\{ x_ix_j|(i,j)\not =(1,2),\ 1\le i<j\le n\}. \]
\end{prop}
\begin{proof}
The degree $3$ integral cohomology $mod(p)$ 
$H^3(X)/p$  is generated as  
a $\bZ/p[y_1....,y_n]$-module by $Q_0(x_ix_j)$. 
The proposition follows from the main lemma and 
\[ Q_1Q_0(x_ix_j)=y_i^{p}y_j-y_iy_j^p\not =0 \quad in\ H^*(X;\bZ/p).\]
\end{proof}

Bogomolov-Bohning-Pirutka study  the kernel
of the map
\[ K=Ker( q^*_{/N^1}:H^*(BV_n;\bZ/p)/N^1\to H^*(BG;\bZ/p)/N^1).\]
where $H^*(BV_n;\bZ/p)/N^1\cong \Lambda(x_1,....,x_{2n})$.
Their theorem in [Bo-Bo-Pi] induces
\begin{thm} (Theorem 1.3 in [Bo-Bo-Pi]) If $p\ge n$, $G$ is extraspecial group of order $p^{1+2n}$  then
$Ker(q^*_{/N^1})\cong  (f)$.  
Hence 
\[  H^*_{ur}(X;\bZ/p)\supset \Lambda(x_1,...,x_{2n})/(f).\]
\end{thm} 
{\bf Remark.}  There is the another group
$p_-^{1+2n}$ with the degree $2^n+1$.

When $p=2$, the situation becomes  changed.
The extraspecial $2$-group $D(n)=2_+^{1+2n}$ in the $n$-th central extension of the 
dihedral group $D_8$ of order $8$.  It has the central extension
\[ 0\to \bZ/2\to D(n)\to V_n\to 0\]
with $V_n=\oplus ^{2n}\bZ/2$.  Hence $H^*(BV_n;\bZ/2)\cong \bZ/2[x_1,...,x_{2n}]$.
Then using the Hochschild-Serre spectral sequence,  Quillen  proved [Qu]
\[H^*(BD(n);\bZ/2)\cong \bZ/2[x_1,...,x_{2n}]/(f,Q_0(f),...,Q_{n-2}(f))\otimes \bZ[w_{2^n}(\Delta)].\]
Here  $w_{2^n}(\Delta)$) is the Stiefel-Whitney class of
 $2^n$-dimensional (spin) representation $\Delta$ which restricts nonzero on the center.
Moreover Quillen proves the
following two theorems (Theorem 5.10-11 in [Qu])
\begin{thm} ([Qu]) $H^*(BD(n);\bZ/2)$ is detected by the product of cohomology of maximal elementary abeian groups.
\end{thm}
\begin{thm} ([Qu]) The nonzero Stiefel-Whitney $w_i(\Delta)$ are those 
of degrees $2^n$ and $2^n-2^i$ for $0\le i<n$. 
\end{thm}
In fact $w_i(\Delta)$ generates the Dickson algebra 
in the cohomology of the maximal elementary abelian $2$-groups.
\begin{prop} When $n>2$, there is the surjection  
\[ \Lambda(x_1,...,x_{2n})/(f)\to H^*(BD(n);\bZ/2)/(N^1).\]
\end{prop}
\begin{proof}
By the same arguments with $p=odd$, we see
\[\Lambda=\Lambda(x_1,...,x_{2n})/(f)\subset H^*(BD(n);\bZ/2).\]
The fact $w_2(\Delta)=0$ follows from the above third 
Quillen's theorem.
Hence we have $w_{2^n}(\Delta)\in N^1$ from Becher's theorem (Theorem 6.2 in [Te-Ya2]). i.e., $w_i$ is 
multiplicative generated by $w_1$ and $w_2$.
Thus we get the proposition. \end{proof}

However this map (in Proposition 9.8) is not need injective.  
 In fact,  in [Bo-Bo-Pi], it is proven that the above map is 
not injective when $G=D(3)=2^{1+6}_+$.  They also see
that the map in the proposition is injective
when we restrict
the degree $*=2$
\begin{thm} Let $G=D(3)$ and $X=X_G$.  
  Then we have
\[ DH^3(X)/2
\cong
 \bZ/p\{Q_0(x_ix_j)|(i,j)\not =(1,2),\ 1\le i<j\le 3\}. \]
\[ H^2_{ur}(X;\bZ/p),\ Inv^2(G:\bZ/p)\supset 
 \bZ/p\{ x_ix_j|(i,j)\not =(1,2),\ 1\le i<j\le 3\}. \]
However, the map in Proposition 9.8 is not  injective for some $*>2$.
\end{thm}

\section{ The motivic cohomology of quadrics over $\bR$ with coefficients $\bZ/2$}

Let $X$ be a smooth variety over the field
 $\bR$ of real numbers,
and we consider the cohomologies of $\bZ/2$ coefficients.
In this paper the $mod(2)$ \'etale cohomology means the 
motivic cohomology of the same first and the second degrees
$ H_{\acute{e}t}^*(X;\bZ/2)\cong H^{*,*}(X.\bZ/2)$.

It is well known ([Vo1], [Vo2])
\[ H_{\acute{e}t}^*(Spec(\bC);\bZ/2)\cong \bZ/2, \quad 
 H^{*,*'}(Spec(\bC);\bZ/2)\cong \bZ/2[\tau], \]
\[ H_{\acute{e}t}^*(Spec(\bR);\bZ/2)\cong \bZ/2[\rho], \quad 
 H^{*,*'}(Spec(\bR);\bZ/2)\cong \bZ/2[\tau, \rho] \]
where $0\not=\tau\in H^{0,1}(Spec(\bR);\bZ/2)\cong \bZ/2$ and
where 
\[\rho=-1\in \bR^*/(\bR^*)^2\cong K_1^M(\bR)/2\cong H_{\acute{e}t}^1(Spec(\bR);\bZ/2).\]

We recall the cycle map from the Chow ring to the \'etale cohomology
\[cl/2: CH^*(X)/2\to H^{2*}_{\acute{e}t}(X;\bZ/2).
\]
This map is also written as
$ H^{2*,*}(X;\bZ/2)\stackrel{\times \tau^*}{\to}
H^{2*,2*}(X;\bZ/2).$

Let $X=Q^d$ be an anisotropic quadric of dimension $2^n-1$
(i.e. the norm variety for ($\rho^{n+1}\in K_{n+1}^M(\bR)/2$)).  Then we have the Rost motive  $M\subset Q^d$ [Ro].    
 It is known ( the remark page 575 in [Ya2])
\[ H_{\acute{e}t}^{*}(M;\bZ/2)\cong \bZ/2[\rho]/(\rho^{2^{n+1}-1})
\cong \bZ/2\{1,\rho,\rho^2,...,\rho^{2^{n+1}-2}\}.\]

The 
Chow ring is also known [Ro]
\[CH^*(M
)/2\cong \bZ/2\{1,c_0,c_1....,c_{n-1}\}, 
\quad cl(c_i)=\rho^{2^{n+1}-2^{i+1}}.\]
The cycle map $cl/2$ is injective.
The elements $c_i$ is also written as 
\[ c_i=\rho^{2^{n+1}-2^{i+1} }
\tau^{-2^n+2^{i}}  
\quad  in \ CH^*(M)/2\subset H_{\acute{e}t}^{2*}(M:\bZ/2)[\tau^{-1}] \]

The mod(2) motivic cohomology is known 
(Theorem 5.3 in \cite{Ya2}).

\begin{thm} (Theorem 5.3 in \cite{Ya2})
  The cohomology 
$ H^{*,*'}(M_n;\bZ/2)$ is isomorphic to the  $\bZ/2[\rho,\tau]$-subalgebra of
\[ \bZ/2[\rho,\tau,\tau^{-1}]/(\rho^{2^{n+1}-1}) \]
generated by  \ \ $ a=\rho^{n+1} ,\  \ a'=a\tau^{-1},\  $  and \ elements in $  \Lambda (Q_0,...,Q_{n-1})\{a'\}.   $
\end{thm}
The following lemma is used in the next sections.
\begin{lemma}   We have 
$Q_0(\tau^{-1})=\rho \tau ^{-2}$.  Hence  $Q_0(a')=\rho a\tau^{-2}, $  while $Q_0(a)=0$.
\end{lemma}
\begin{proof}
We see the first equation from
\[ 0=Q_0(1)=Q_0(\tau \tau^{-1})=\rho \tau^{-1}+\tau Q_0(\tau^{-1}).\]
\end{proof}

\begin{lemma} (Lemma 5.13 in \cite{Ya2}
)   Let $X_d$ be anisotropic quadric of the degree $d$.
Let $h\in H^{2,1}(X_d)$ be the hyper plain section.
If $2^n-2<d$, then we have a graded ring isomorphism
\[H^{*,*'}(X_d;\bZ/2)\cong \bZ/2[\rho,\tau,h]\quad 
when \ \ *\le n.\]
In particular, $H^{*,*-1}(X_d:\bZ/2)=0$ 
$mod(Ideal(h))$ for $*\le n$.
\end{lemma}

\section{The cohomology of quadrics with coefficients in $\bZ_2$}

In this section we consider integral coefficients case.
In this paper, the $2$-adic  integral $\bZ_2$ cohomology means the inverse limit
\[ H_{\acute{e}t}^*(M;\bZ_2)= 
Lim_{\infty \gets s
}H^{*,*}(M;\bZ/2^s)\]
of motivic cohomologies.

We recall here the Lichtenberg cohomology
[Vo1,2]
such that 
\[ H_L^{*,*'}(X;\bZ)\cong H^{*,*'}(X;\bZ)\quad for\ *\le *'+1.\]
(The right side is the motivic cohomology.)
By the five lemma, we see (for $1/s\in k$)
\[ H_L^{*,*'}(X;\bZ/s)\cong H^{*,*'}(X;\bZ/s)
\quad  for\ \ *\le *'.\]
Moreover we have
$H_L^{2*,*'}(X;\bZ/s)\cong
H^{2*}_{\acute{e}t}(X;\mu_{s}^{* '\otimes })$.


In this paper we consider the cycle maps to this Lichitenberg
(or motivic) cohomology in stead of  the \'etale cohomology itself.  The cycle map is written
\[ cl: CH^*(X)\otimes \bZ_2\cong H^{2*,*}(X;\bZ_2 )\to H_L^{2*,*}(X;\bZ_2)\cong H_{\acute{e}t}^{2*}(X;\bZ_2(*))\]
where $\bZ(*)$ is the Galois module, when $k=\bR$,
 it acts as $(-1)^{*}$.  
Here we can write  \ \ 
\[H^{2*}_{\acute{e}t}(X;\bZ_2(*))
=\oplus_{m\ge 0}(H^{4m}_{\acute{e}t}X;\bZ_2)\oplus
H^{4m+2}_{\acute{e}t}(X;\bZ_2(1)).  \]
Note that it is the (graded) ring.

Let $k=\bR$. Moreover let $*=even$.  Then the right hand side cohomology is written
\[ H^{2*}_{\acute{e}t}(X;\bZ_2(*))\cong  H^{2*}_{\acute{e}t}(X;\bZ_2(even))\cong
H^{2*}_{\acute{e}t}(X;\bZ_2(2*))\]
\[ \cong H^{2*,2*}_{L}(X;\bZ_2)\cong H^{2*,2*}(X;\bZ_2).\]
Similarly, when $*=odd$, we see
$ H^{2*}_{\acute{e}t}(X;\bZ_2(*))\cong 
 H^{2*,2*+1}(X;\bZ_2).$

Thus in this paper, the cycle map means ;
\[ cl: CH^*(X)\otimes \bZ_2\to H_{\acute{e}t}^{2*}(X;\bZ_2(*))\cong \begin{cases}
   H^{2*,2*}(X;\bZ_2)\quad for\ *=even\\
 H^{2*,2*+1}(X;\bZ_2)\quad for\ *=odd.\end{cases}.
\]
We say that $x\in H_{\acute{e}t}^{2*}(X;\bZ(*)$ is non-algebraic if $x\not =0 \ mod(Im(cl))$.

The short exact sequence $0 \to \bZ \stackrel{2}{\to} \bZ\to \bZ/2\to 0$ induces the  long exact sequence of motivic cohomology
\[ ... \to H^{*-1,*}(M;\bZ/2)\stackrel{\delta}{\to} H^{*,*}(M;\bZ)
    \stackrel{2}{\to}
H^{*,*}(M;\bZ)\stackrel{r}{\to}
 H^{*,*}(M;\bZ/2) 
\to ... \] 
By Voevodsky [Vo1], [Vo2]), it is known $\beta(\tau)=\rho$
for the Bockstein operation $\beta$.
Let us write  $\delta (\tau \rho^{i-1})=\bar \rho_i\in 
H^{*,*}(M;\bZ)$ so that 
$r(\bar \rho_i)=\beta(\tau \rho^{i-1})=\rho^i$
since $r\delta=\beta$. 
Moteover  $\bar \rho_i$ is $2$-torsion from the above exact sequence.

Hence for all $1\le c\le 2^{n+1}-2$, we see 
$H^{c,c}(M;\bZ)\not =0$.
The same fact holds each
$H^c_{\acute{e}t}(M;\bZ/2^s)$ and so $H^c_{\acute{e}t}(M;\bZ_2)$.

\begin{lemma} Let $N=2^{n+1}-2$.  Then
\[ \bZ_2\{1,cl(c_0)\}\oplus \bZ/2\{\bar \rho_1,...,\bar \rho_{N}\}\subset 
H^*_{\acute{e}t}(M;\bZ_2)\oplus H^*_{\acute{e}t}(M;\bZ_2(1)).\] 
The element $\bar \rho_c$ with 
$c=0\ mod(4)$ 
 and $c\not = 2^{n+1}-2^{i+1}$ is  a non-algebraic element
(i.e., not in the image of the cycle map).
\end{lemma}

{\bf Remark.}
When $c=2\ mod(4)$, the element
$ \bar \rho_c \in H^c(M;\bZ_2)$ but not in 
$H^c(M;\bZ_2(1)).$  So we identify here $\bar \rho_c$ is not in
$H_{\acute{e}t}^{2*}(M;\bZ_2(*)).$

Writing $\pi=cl(c_0)$, we have the following theorem.
\begin{thm} ([Ya7])
Let $M_n\subset Q^{2^n-1}$ be the Rost motive of the 
norm variety.  Then there are element 
$\pi\in H^{2^{n+1}-2}_{\acute{e}t}(M_n; \bZ_2(1))$ and 
$\bar \rho_{4m}\in
H^{4m}_{\acute{e}t}(M_n;\bZ_2(0))$ such that 
\[ H^{2*}_{\acute{e}t}(M_n;\bZ_2(*))\cong
\bZ_2\{1,\pi\}\oplus \bZ/2\{\bar \rho_4,
\bar \rho_8,...,\bar \rho_{2^{n+1}-4} \}\]
\[ \cong
\bZ_2\{1,\pi\}\oplus \bZ/2[\bar \rho_4]^+/(
\bar \rho_4^{2^{n-1}})
.\]
The image of the cycle map is given
\[CH^*(M_n)\otimes \bZ_2\cong
\bZ_2\{1,\pi\}\oplus \bZ/2\{\bar \rho_{2^{n+1}-2^{n}},
\bar \rho_{2^{n+1}-2^{n-1}},...,\bar \rho_{2^{n+1}-4} \} .\]
\end{thm}

\section{norm varieties}

Let $X=Q^{2^n-1}$ be the norm variety, and $M_n$
be its Rost motive.  We have the decomposition of motives
(\cite{Ro}, $\S 6$ in[Ya1])
\[M(Q^{2^n-1})\cong M_{n}\oplus M_{n-1}\otimes M(\tilde \bP^{2^{n-1}-1})
\]
where  \ $M(\tilde \bP^s)\cong \bT\oplus...
\oplus \bT^{s\otimes}$.

Hence we have the additive structure from Theorem 11.2
in the preceding section.  More strongly, we can prove

\begin{thm}  \cite{YaC}
 We have a  ring isomorphism 
\[ H^{2*}_{\acute{e}t}(Q^{2^n-1};\bZ_2(*))
\cong \bZ_2[h,\bar \rho_4]/(h^{2^n}, 2\bar \rho_4, h\bar \rho_4^{2^{n-2}}, \bar \rho_4h^{2^{n-1}},
 \bar \rho_4^{2^{n-1}}).\]
Here $h\in H^2(Q^{2^n-1};\bZ_2(1))$ is the hyper plain section, and we can take  $\pi=h^{2^n-1}$. (The ring is generated by only two elements.)
\end{thm}

We give only an outline of the proof  for $Q^7$ here, for ease of arguments.

\begin{lemma}  We have a  ring isomorphism 
\[ H^{2*}_{\acute{e}t}(Q^7;\bZ_2(*))
\cong 
\bZ_2[h]/(h^8)
\oplus \bZ/2[h]/(h^4)\{\bar \rho_4\}\otimes \bZ/2\{\bar \rho_4^2,\bar \rho_4^3\}\]
\[  \cong \bZ_2[h,\bar \rho_4]/(h^8, 2\bar \rho_4,  
h^4\bar \rho_4, h\bar \rho_4^2, \bar \rho_4^4)\]
where $h^7=c_0=\pi, \ c_1=\bar \rho_4^3, \ c_2=\bar \rho_4^2$
and $c_1'h=h\bar \rho_4$.
Hence we have 
\[H^{2*}_{\acute{e}t}(Q^7;\bZ_2(*))/(Im(cl))\cong \bZ/2\{\bar \rho_4\}.\]
\end{lemma}
\begin{proof}
From the decomposition of
the motive, we see (additively)
\[ H^{2*}_{\acute{e}t}(Q^7;\bZ_2(*))\cong 
H^{2*}(M_3;\bZ_2(*))\oplus H^{2*}(M_2;\bZ_2(*))\otimes \bZ_2\{h,h^2,h^3\}
.\]
Hence it can be written additively  (with $ |c_0|=14$, $|c_1|=12$, 
$|c_2|=8,$ 
$|c_0'|=6$, $|c_1'|=4$) 
\[ (\bZ_2\{1,c_0\}\oplus \bZ/2\{\bar \rho_4,c_1,c_2\})
\oplus (\bZ_2\{1,c_0'\}\oplus \bZ/2\{c_1'\})\otimes 
\bZ_2\{h,h^2,h^3\}.\]

It is well known  (for $\bar X=X(\bC)$)
\[H^*(\bar X;\bZ_2)\cong \bZ_{2}[h,y]/(h^8,2y=h^4,y^2). \]
Hence, from the restriction map,  the ring $H^*(X;\bZ_2)
\supset Z_2[h]/(h^8)$.


First note 
\[\bZ_2\{c_0'h,c_0'h^2,c_0'h^3\}\cong \bZ_2\{h^4,h^5,h^6\}. \]
Thus we have 
\[ \bZ_2\{1,h,...,h^7\}\cong \bZ_2\{1,h,h^2,h^3,hc_0',h^2c_0',h^3c_0',c_0\}.\]
So we have the above $H^{2*}(Q^7;\bZ_2(*))$ is isomorphic to
\[ \bZ_2[h]/(h^8)\oplus \bZ/2\{\bar \rho_4,c_2,c_1\}\oplus
\bZ_2\{c_1'h,c_1'h^2,c_1'h^3\}.\]
Taking $c_2=\bar \rho_4^2, \ c_1=\bar \rho_4^3,\ hc_1'=h\bar \rho_4$,
we have the result.
\end{proof}

We want to see the following theorem.

\begin{thm}
Let $X_n=Q^{2^n-1}$, $n\ge 2$ the norm variety.  Then
\[ DH^{2*}(X_n;\bZ_2(*))=0.\]
\[H_{ur}^{2*}(X_n;\bZ_2(*))\supset  \bZ/2[\bar \rho_4]/(
\bar\rho_4^{2^{n-1} }).\]
Hence for $n\not =n'$, we see that $X_n$ and $X_{n'}$ are not retract rationally equivaliant.
\end{thm}
{\bf Remark.}  When $n=1$, we see $X_1\cong \bP^1$ that is,
$X_1$ stable birational.

\begin{cor}
The second and the last formulas in the above theorem, hold when $k$ is a real number field.
\end{cor}
\begin{proof} 
 Recall that the norm variety $X_n=X_n(k)$ is defined naturaly.

Let $r_1\ge  1$ be real embedding number.
Then we have the restriction
\[r : H_{et}^*(Spec(k);\bZ/2) \to \oplus^{r_1}H_{et}^*(Spec(\bR  );\bZ/2) \]
such that $r$ is surjective for $*\ge 1$ and isomorphic
for $*\ge 3$,  Hence we can define $\rho(k)\in H_{et}^1(Spec(k);\bZ/2)$ so that
$r(\rho(k))=\rho$.  Similarly we can define $\bar \rho_4(k)$
in  $ H_{un}^*(X;\bZ_2)$.  It is nonzero since so for $k=\bR$.

\end{proof}

Since the hyper plain section $h$ is represented by
a first Chern class and $\pi=c_0= h^{2^n-1}$.  By Frobenius reciprocity, we only check elements
\[ \bar \rho_4^i \not  \in N^1H^{2*}(X;\bZ_2(*))  \]
for the first equality in theorem. 
Then $Ideal(h)\in N^1$ and 
\[H^{4*}_{ur}(X;\bZ/2)\supset H^{4*}(X:\bZ/2)/N^1
\supset   \bZ/2[\bar \rho_4]
/(
\bar\rho_4^{2^{n-1} })\]
implies the second formula.

Since $H^{2*}(M_n;\bZ_2(*))$ is a direct summand of 
$H^{2*}(X;\bZ_2(*))$ and $\bar \rho_4$ is defined in
$H^{2*}(M_n;\bZ_2(*))$, we only need to see the following 
Lemma 12.5, (by using Lemma 10.1-10.3) for
the proof of the above theorem.

\begin{lemma}
We have $\bar p_4^i\not \in N^1H^*(M_n;\bZ_2(*))$ for $i\ge 1$.  
\end{lemma}
\begin{proof}Consider the following diagram

\[ \begin{CD}
\bar \rho_s\in H^{*,*}(X;\bZ_2) @>{r}>>
\rho^s\in H^{*,*}(X:\bZ/2)  @.\\
@A{\tau'}AA     @A{\tau}AA \\
 x\in H^{*,*-1}(X:\bZ_2) @>{r}>>
H^{*,*-1}(X;\bZ/2) @>{Q_0}>> 
\end{CD}
\]

Suppose $\bar \rho_s\in N^1H^{*,*}(X;\bZ_2)$, which  means that there is 
 $x\in H^{*,*-1}(X:\bZ_2)$ such that $\tau'x=\bar \rho_s$.
We consider the reduction maps $r$ to the cohomology of $\bZ/2$ cefficients.
  Then $\tau r(x)=\rho ^s$.  and  $Q_0(r(x))$
must be zero (since $x$ is in the integral coefficients $\bZ_2$).
We will prove this does not happen.

Recall $a=\rho^{n+1}$ and $a'=a\tau^{-1}$ in $H^{*.*-1}(M_n;\bZ/2)$.

The case $*\le n$ ; 
The cohomology $H^{*,*'}(M_n;\bZ/2)=0$ $mod(Ideal(h))$
for $*>*'$ from Lemma 10.3.     Hence there is no non zero element 
 $\tau^{-1}\rho^* \in H^{*.*-1}(M_n;\bZ/2) \ mod(Ideal(h)$
(where $h\in \tilde N^1$).

The case $*=n+1$;    
Then there is $a'$ such that $\tau a'=a$.
But this element $a'$ is not in the integral $H^{*.*-1}(M_a;\bZ_2)$, because 
\[Q_0(a')=Q_0(\rho^{n+1}\tau^{-1})=\rho^{n+2}\tau^{-2}.\]
which is nonzero in  $H^{*,*'}(M_n;\bZ/2),$ and 
so $a\not\in N^1 H^{*,*}(M_n;\bZ_2)$. 

The case $*>n+1$ .
Let us write $b'=Q_0(a')=\rho^{n+2}\tau^{-2}$.
Next consider the element $b=\tau b'$.  Then we note
\[ \tau b=\tau^2b'= \tau^2 \rho^{n+2}\tau ^{-2}
=a\rho=\rho^{n+2}. \]
That is $b=\rho^{n+2}\tau^{-1}$and $b=\tau Q_0(a')$.
Hence from Theorem 10.1, we see $b\in H^{*.*'}(M_n;\bZ/2)$.

Since $Q_0(b')=Q_0Q_0(a')=0$, we can compute
\[Q_0(b)=Q_0(\tau b')=\rho b'+\tau Q_0(b')=\rho b' \]
is nonzero in $H^{*,*'}(X;\bZ/2)$ and hence $b$ is
not in the integral $H^{*,*'}(X;\bZ_2)$.  Therefore 
$\rho^{n+2}  \not  \in N^1H^*(X;\bZ_2). $

Similarly we can show for $j>n+2$, the element $\rho^{j}$ is not in
$N^1H^*(X;\bZ_2). $
\end{proof}

The elements $a,...,b'$ are written in $\bZ/2[\rho,\tau,\tau^{-1}]/(\rho^{2^n-1})$ as follows.
(Recall Theorem 10.1 and Lemma 10.2.)

\[\begin{CD}
   @.  @.  \rho^{n+2}\in H^{n+2,n+2}  @.  \\
   @.   @.  @A{\tau}AA   @.  \\
   @.   a=\rho^{n+1}\in H^{n+1,n+1}\   @.  b =\rho^{n+2}\tau^{-1}@>{Q_0}>> \rho^{n+3}\tau ^{-2} \\
  @.    @A{\tau}AA    @A{\tau}AA   @.  \\
   \rho^n\in H^{n.n}  @. \quad a'=\rho^{n+1}\tau^{-1}   @>{Q_0}>> b'=\rho^{n+2}\tau^{-2} @>{Q_0}>>  0  \\
@A{\tau}AA @.  @. @. @.  \\
0=H^{n,n-1}  @.  @. @. @.
 \end{CD} 
\]

\end{document}